\documentclass[a4paper,10pt]{article}
\usepackage{amsmath,amssymb}
\usepackage[utf8]{inputenc}
\usepackage[colorlinks,citecolor=blue]{hyperref}
\usepackage[french]{babel}
\usepackage{booktabs}
\usepackage[mathcal]{euscript}
\usepackage{graphicx}

\title{Une note sur les intervalles de Tamari}
\author{F. Chapoton}
\date{\today}

\newtheorem{theorem}{Théorème}[section] 
\newtheorem{proposition}[theorem]{Proposition} 
\newtheorem{conjecture}[theorem]{Conjecture} 
 
\newtheorem{lemma}[theorem]{Lemma}

\newtheorem{remark}[theorem]{Remarque}

\newenvironment{proof}{\begin{trivlist}\item{\bf{Preuve.}}}
  {\hfill\rule{2mm}{2mm}\end{trivlist}}

\newcommand{\Int}{\operatorname{Int}}
\newcommand{\Tam}{\operatorname{\mathbf{Tam}}}
\newcommand{\quat}{\mathbb{D}}
\newcommand{\comp}{\mathcal{C}}

\newcommand{\QQ}{\mathbf{Q}}
\newcommand{\couvert}{\triangleleft}

\newcommand{\LL}{\mathsf{LL}}
\newcommand{\LR}{\mathsf{LR}}
\newcommand{\RR}{\mathsf{RR}}

\newcommand{\ab}{\overline{a}}
\newcommand{\xb}{\overline{x}}
\newcommand{\yb}{\overline{y}}

\newcommand{\gch}{\includegraphics[height=4mm]{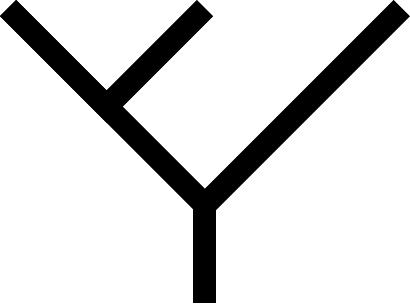}}
\newcommand{\drt}{\includegraphics[height=4mm]{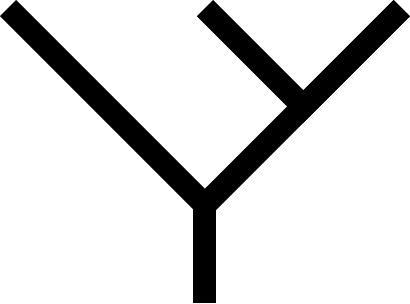}}

\begin{document}
\maketitle

\begin{abstract}
  À tout ordre partiel $P$, on associe un polynôme $\quat_P$ en quatre
  variables, qui énumère les intervalles dans $P$ en fonction
  de quatre paramètres dont la description utilise un ordre partiel
  naturel sur les intervalles.

  On s'intéresse aux symétries de cet invariant général lorsqu'il est
  appliqué à une famille importante d'ordres partiels, les treillis de
  Tamari. On obtient une symétrie ternaire pour une spécialisation du
  polynôme (en utilisant une équation fonctionnelle et une équation
  algébrique pour la série génératrice) et une conjecture sur une
  symétrie globale du polynôme. On décrit le sous-ensemble connu des
  intervalles synchrones des treillis de Tamari en terme d'une facette
  dans un polytope de Newton. On relie une autre spécialisation du
  polynôme aux statistiques provenant de la canopée des arbres
  binaires plans.
\end{abstract}

% Pour le cas $m$-Tamari, le polynôme $\quat$ a beaucoup moins de
% symétries. Ces posets ne sont pas auto-duaux en général.

% Il semble rester néanmoins une forte symétrie non-triviale entre les
% variables $y$ et $\xb$, avec un double-distribution symétrique.

% Il semble aussi qu'on ait une (autre) distribution commune pour $x$ et
% $\yb$, mais sans double-distribution symétrique.

\section{Polynômes de valence et posets d'intervalles}

Soit $P$ un poset (ensemble muni d'un ordre partiel) fini. On peut lui associer un polynôme en $2$
variables 
\begin{equation*}
  D_P(a,\ab) = \sum_{u \in P} a^{\operatorname{out}(u)} \ab^{\operatorname{in}(u)},
\end{equation*}
qui compte les sommets de $P$ selon leur valence
sortante (pour $a$) et entrante (pour $\ab$) dans le diagramme de Hasse
de $P$. On convient d'orienter les arêtes du diagramme de Hasse dans
le sens croissant pour le poset : $u \to v$ implique $u \leq v$. On va aussi utiliser la notation $u \couvert v$ pour la relation de couverture.

Par exemple, pour le treillis de Tamari $\Tam_3$ ayant un diagramme de
Hasse en forme de pentagone, on trouve
\begin{equation*}
  D_{\Tam_3}(a,\ab) = a^2 + 3 a \ab + \ab^2.
\end{equation*}

Pour un produit cartésien de posets $P \times Q$, le polynôme $D_{P \times Q}$ est le produit des polynômes $D_P$ et $D_Q$. Le polynôme $D_{P^\star}$ pour le poset dual $P^\star$ est
\begin{equation*}
  D_{P^\star}(a, \ab) = D_P(\ab, a).
\end{equation*}
Le polynôme $D_P$ est donc symétrique en $a$ et $\ab$ lorsque le poset $P$ est auto-dual.

En général, le polynôme $D_P$ n'est pas homogène et ne présente pas de
régularité visible au niveau de son support. Par exemple pour l'ordre
de Bruhat sur le groupe symétrique, rien ne saute
aux yeux.

\bigskip

À chaque poset $P$, on peut par ailleurs associer le poset de ses
intervalles $\Int(P)$. C'est un ordre partiel sur l'ensemble des
paires $(u,v)$ avec $u \leq v$ dans $P$, donné par la relation
$(u,v) \leq (u', v')$ si et seulement si $u \leq u'$ et $v \leq
v'$. Les couvertures d'un intervalle $(u,v)$ sont soit de la forme
$(u',v)$ pour $u'$ une couverture de $u$ dans $P$ qui reste inférieure
à $v$, soit de la forme $(u,v')$ pour une couverture quelconque $v'$ de $v$ dans $P$.

Le poset des intervalles $\Int(P^\star)$ du poset dual $P^\star$ est le dual du
poset des intervalles de $P$. Le poset des intervalles d'un produit
cartésien $P \times Q$ est le produit cartésien des posets
d'intervalles de $P$ et de $Q$.

\medskip

Dans le cas d'un poset d'intervalles, on peut raffiner le polynôme
$D_{\Int(P)}$ en introduisant quatre variables $x,y,\yb,\xb$ qui correspondent respectivement à quatre types d'arêtes incidentes au sommet $(u,v)$ dans le diagramme de Hasse de $\Int(P)$ :
\begin{itemize}
\item $x$ : arête sortante $(u,v) \couvert (u',v)$ avec $u \couvert u' \leq v$,
\item $y$ : arête sortante $(u,v) \couvert (u,v')$ avec $v \couvert v'$,
\item $\yb$ : arête entrante $(u',v) \couvert (u,v)$ avec $u' \couvert u$,
\item $\xb$ : arête entrante $(u,v') \couvert (u,v)$ avec $u \leq v' \couvert v$.
\end{itemize}
On définit donc ainsi un polynôme $\quat_P(x,y,\yb,\xb)$ en $4$
variables, comme la somme sur l'ensemble des intervalles des monômes
$x^\star y^\star \yb^\star \xb^\star$ dont les puissances décrivent
les arêtes entrantes et sortantes pour un intervalle $(u,v)$.

Le polynôme associé au poset dual $P^\star$ est donné par
\begin{equation*}
  \quat_{P^\star}(x,y,\yb,\xb) = \quat_{P}(\xb,\yb,y,x).
\end{equation*}
Lorsque le poset $P$ est auto-dual, le poset $\Int(P)$ l'est aussi et
le polynôme $\quat_{P}(x,y,\yb,\xb)$ est invariant sous l'involution
$x \leftrightarrow \xb$ et $y \leftrightarrow \yb$.

Pour un produit cartésien de posets $P \times Q$, le polynôme $\quat_{P \times Q}$ est le produit des polynômes $\quat_P$ et $\quat_Q$.

On retrouve bien sûr le polynôme plus simple en deux variables $a,\ab$
par la substitution
\begin{equation*}
  D_{\Int(P)}(a,\ab) = \quat_{P}(a,a,\ab,\ab).
\end{equation*}

\textbf{Convention :} par abus de notation, on dira qu'un intervalle
$(u,v)$ dans $P$ est de degré $d$ par rapport à un sous-ensemble de
$\{x,y,\yb,\xb\}$ si le monôme correspondant dans $\quat_{P}$ est de
degré total $d$ par rapport à cet ensemble de variables.

\section{Le cas des treillis de Tamari}

Les treillis de Tamari et leurs intervalles ont fait l'objet de
nombreuses recherches récentes, voir entre autres articles \cite{CCP,chapoton_SLC,BM_et_al1,BM_et_al2,bebo,Fang_PR,PR_viennot} et le volume \cite{livre_tamari}.

\label{section-tamari}

On va utiliser les conventions suivantes pour les treillis de Tamari
$\Tam_n$. Le treillis $\Tam_n$ est vu comme un ordre partiel sur les
arbres binaires plans avec $n$ sommet internes et $n+1$ feuilles,
qu'on va dessiner comme dans la figure ci-dessous. On utilisera le mot
\textit{sommet} pour signifier sommet interne. Les arbres binaires
plans ont leur racine (qui est un sommet interne) en bas et croissent
vers le haut. L'arbre minimum est le peigne droit (dont aucun sommet
n'a de fils gauche), l'arbre maximum est le peigne gauche (dont aucun
sommet n'a de fils droit). La relation d'ordre est la clôture
transitive de la relation de couverture, qui est la rotation des
arbres binaires plans (dans le sens $\drt \to \gch$). On rappelle que le poset $\Tam_n$ est
auto-dual, par le biais du renversement gauche-droite des arbres
binaires. La \textit{canopée} d'un arbre binaire à $n+1$ feuilles est
un suite de $n+1$ lettres $L$ ou $R$, chaque lettre décrivant
l'orientation d'une feuille, les feuilles étant considérées de gauche
à droite. Par exemple, l'arbre binaire
\begin{center}
  \label{arbre}
  \includegraphics[height=2cm]{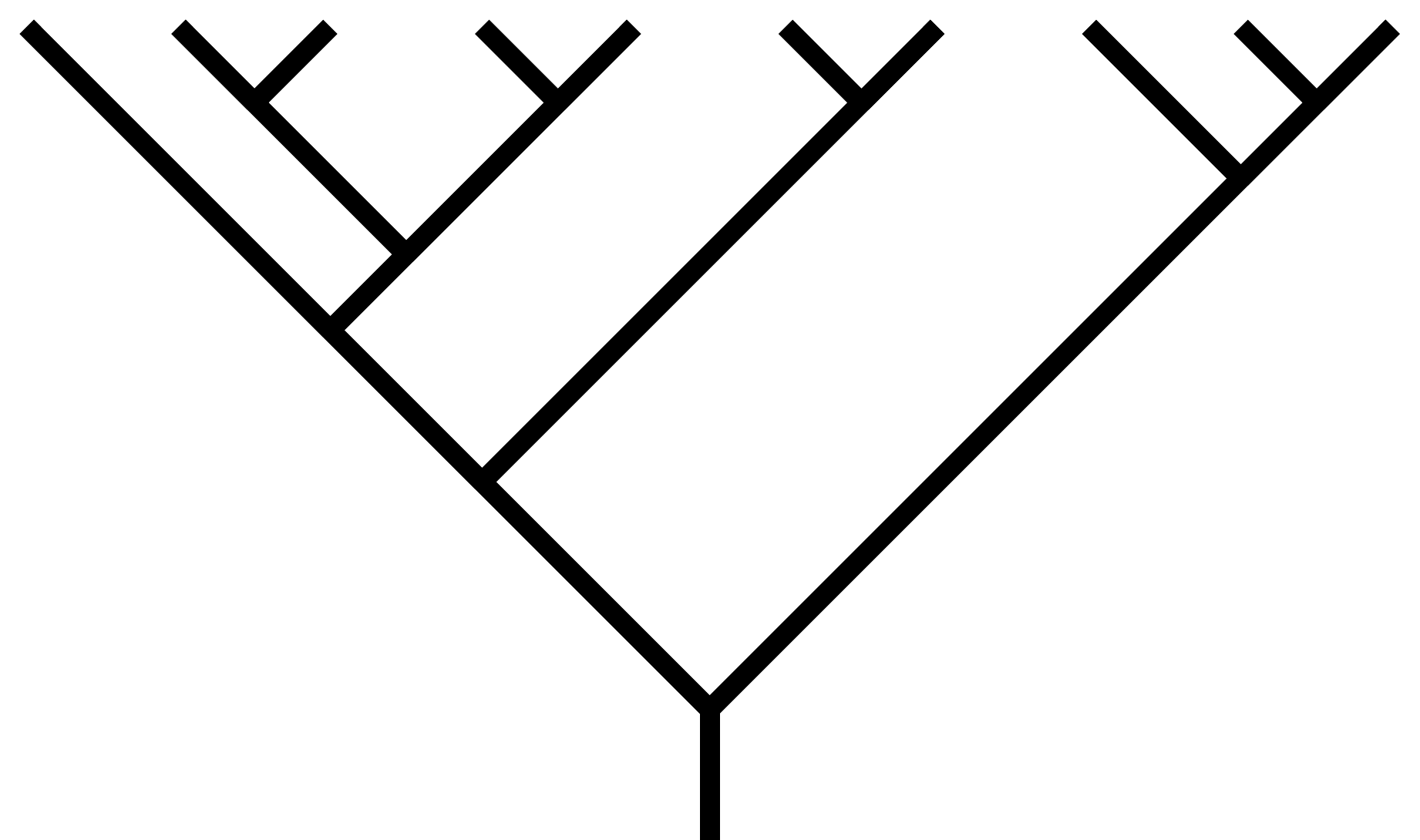}
\end{center}
a pour canopée $(L,L,R,L,R,L,R,L,L,R)$.

\medskip

On observe, en calculant les premiers polynômes
$D_{\Int(\Tam_n)}(a,\ab)$, la propriété remarquable que leur support
(ensemble des monômes) forme un triangle
\begin{equation*}
  \{(i,j) \mid i+j \geq n-1, i \leq n-1, j \leq n-1\}.
\end{equation*}
Voici les premiers de ces triangles pour $n=1,2,3,4,5$ (avec l'origine
des coordonnées en bas à gauche) :
\begin{equation*}
\left(\begin{array}{r}
1
\end{array}\right)\left(\begin{array}{rr}
1 & 1 \\
 & 1
\end{array}\right)\left(\begin{array}{rrr}
1 & 3 & 2 \\
 & 3 & 3 \\
 &  & 1
\end{array}\right)\left(\begin{array}{rrrr}
1 & 6 & 11 & 4 \\
 & 6 & 16 & 11 \\
 &  & 6 & 6 \\
 &  &  & 1
\end{array}\right)\left(\begin{array}{rrrrr}
1 & 10 & 35 & 36 & 9 \\
 & 10 & 50 & 86 & 36 \\
 &  & 20 & 50 & 35 \\
 &  &  & 10 & 10 \\
 &  &  &  & 1
\end{array}\right).
\end{equation*}

Le polynôme raffiné $\quat_n = \quat_{\Tam_n}$ est encore plus remarquable, d'abord par la symétrie ternaire suivante, inattendue.
\begin{theorem}
  \label{principal}
\item[(A)] Le polynôme $\quat_n(x,y,\yb,1)$ est totalement symétrique en $x,y,\yb$.
\item[(B)] Le polynôme $\quat_n(1,y,\yb,\xb)$ est totalement symétrique en $y,\yb,\xb$.
\end{theorem}

Les deux énoncés $(A)$ et $(B)$ sont équivalents par l'invariance
connue sous l'involution (provenant de l'auto-dualité de $\Tam_n$) qui
échange $x$ avec $\xb$ ainsi que $y$ avec $\yb$. Il n'y a aucune
raison \textit{a priori} d'attendre une symétrie totale entre $x$, $y$
et $\yb$. \textit{A posteriori}, on peut soupçonner que la bijection
connue avec les triangulations \cite{bebo} pourrait expliquer cette symétrie
ternaire. La preuve de ce théorème passe par l'obtention d'une
équation algébrique pour la série génératrice des $\quat_n(x,y,\yb,1)$
dans la section \ref{algeb}.

Comme conséquence de ce théorème, les paires de variables
\begin{equation*}
  (x,y), (x,\yb),(y,\yb),(y,\xb), (\yb,\xb)
\end{equation*}
ont toutes la même double-distribution. Dans la section \ref{canopee}, on identifie cette double-distribution avec celle de deux paramètres décrivant les canopées des intervalles de Tamari. La double-distribution pour la
paire $(x,\xb)$ est très différente, presque diagonale.

\medskip

Une autre symétrie de $\quat_n$  reste à démontrer.
\begin{conjecture}
  \label{conjecture_x_xbar}
  Le polynôme $\quat_n$ est symétrique par échange de $x$ et $\xb$. Il
  est aussi symétrique par échange de $y$ et $\yb$.
\end{conjecture}
Les deux parties de l'énoncé sont équivalentes du fait de l'invariance
connue sous l'échange simultané de $x \leftrightarrow \xb$ et
$y \leftrightarrow \yb$. Cet énoncé a été vérifié expérimentalement
jusqu'à $n=8$ inclus. Il échappe aux techniques de preuves utilisées
ici, car on ne dispose pas d'équations catalytiques pour la série
génératrice tenant compte des $4$ variables $x,y,\yb,\xb$.  Il semble
tentant de chercher un involution pour démontrer cette conjecture. Un
calcul pour $n=6$ montre qu'une telle involution ne peut pas préserver
la longueur des chaînes maximales dans les intervalles.

%MEMO : $\yb = c$ et $\xb = d$

\medskip

Un sous-ensemble remarquable des intervalles de $\Tam_n$ est formé par
les intervalles synchrones, qui sont ceux où la canopée du minimum est
égale à la canopée du maximum. Cet ensemble a le même cardinal
que les permutations triables-par-deux-piles dans le groupe symétrique $\mathfrak{S}_n$
\cite{Fang_PR} (voir \href{http://oeis.org/A000139}{A139}).

\begin{theorem}
  \label{synchrone}
  Les intervalles synchrones dans $\Tam_n$ sont exactement les
  intervalles de $\Tam_n$ de degré $n-1$ en $(y,\yb)$.
\end{theorem}
La preuve sera donnée dans la section \ref{preuve_synchrone}.

Par l'égalité des double-distributions mentionnée plus haut, on a donc
au total $5$ sous-ensembles d'intervalles de $\Tam_n$ de même cardinal que les permutations
triables-par-deux-piles, lorsque le degré en
$(x,y),(x,\yb),(y,\yb),(y,\xb)$ ou $(\yb,\xb)$ est $n-1$.

La distribution induite, par exemple celle des variables $(x,y)$ sur
les intervalles de degré $n-1$ en $(x,y)$ semble être identique à celle des
permutations triables-par-deux-piles selon une statistique connue (voir \href{http://oeis.org/A082680}{A82680}).

\section{Séries génératrices}

\label{series_gens}

On va chercher une équation fonctionnelle pour la série génératrice
des $\quat_n(x, y, \yb,1)$, puis montrer que ceci implique que cette série
génératrice est algébrique et enfin observer la symétrie ternaire sur l'équation
algébrique obtenue.

\subsection{Description combinatoire}

On va utiliser comme série génératrice une série formelle en la
variable $t$ (qui compte le nombre de sommets dans les arbres binaires
plans) à coefficients dans $\QQ[u,v,x,y,\yb]$. Ici $u$ et $v$ sont 2
variables catalytiques (auxiliaires), qui sont utiles (et même
nécessaires) pour écrire les équations fonctionnelles.

Tout arbre binaire $S$ dans $\Tam_n$ admet une unique décomposition
par découpage le long de son bord gauche. Plus précisément, $S$ admet
une unique décomposition maximale de la forme
$S_0 / S_1 / \dots / S_\ell$, où l'opération associative $S_i / S_j$ est
la greffe de la racine de $S_i$ sur la feuille la plus à gauche de $S_j$.
On dit que $S$ est \textit{indécomposable} lorsque cette décomposition
est triviale ($\ell=0$).

Soit $S$ un arbre binaire dans $\Tam_n$ avec une décomposition
maximale comme ci-dessus, de la forme $S_0 / S_1 / \dots / S_\ell$. On
peut lui associer une composition $\comp(S)$ qui est la liste
$(n_0,n_1,\dots,n_\ell)$ des nombres de sommets des arbres binaires
$S_i$. Pour l'arbre dessiné au début de la section
\ref{section-tamari}, la composition $\comp(S)$ est $(4,2,3)$.

Si $S \leq T$ dans $\Tam_n$, la composition $\comp(S)$ est plus
grossière que la composition $\comp(T)$. Il suffit de le voir pour une
couverture $S \couvert T$ dans $\Tam_n$ ; dans ce cas la propriété découle de la
description des couvertures par rotation $\drt \to \gch$ des arbres binaires.

En particulier, soit $k_S$ le dernier nombre dans la composition
$\comp(S)$. C'est le nombre de sommets du dernier terme (celui qui
contient la racine) dans la décomposition par découpage le long du bord
gauche de $S$.

Il existe alors une unique partition de l'ensemble des arêtes du bord
gauche de $T$ en une partie $U$ (en bas) et une partie $V$ (en haut)
telle que l'arbre (contenant la racine) obtenu en coupant au milieu de
l'arête inférieure de $V$ ait exactement $k_S$ sommets.

\begin{center}
  \label{arbre}
  \includegraphics[height=2cm]{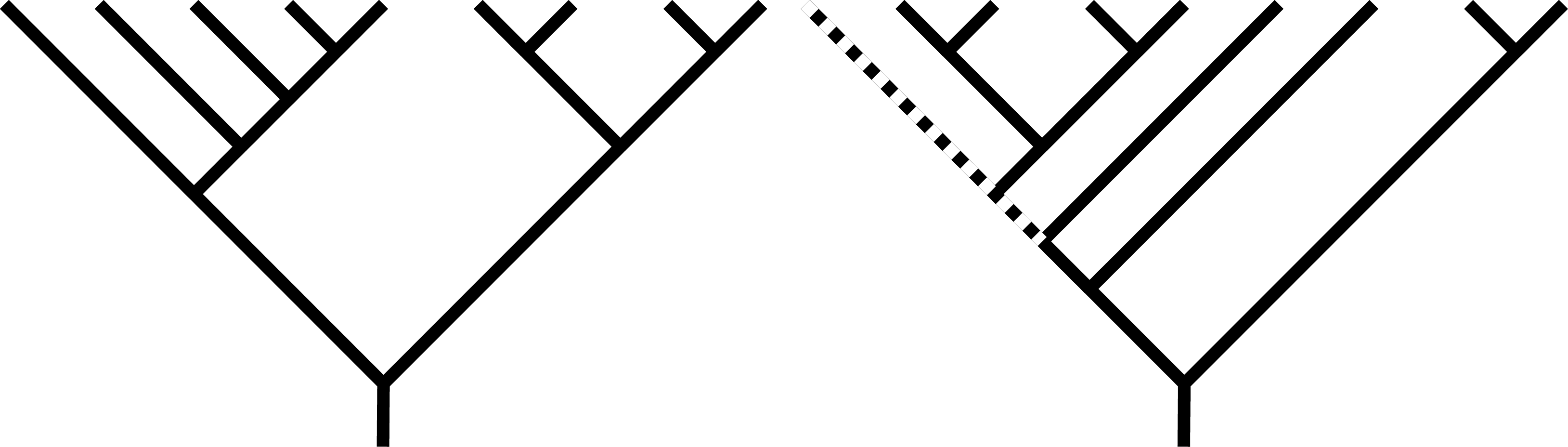}
\end{center}
Voici une illustration de cette description avec $S$ à gauche de
composition $\comp(S)=(4,4)$ et $T$ à droite de composition
$\comp(T)=(4,1,1,2)$, et la partie $V$ du bord de $T$ en pointillés.

On va utiliser la variable $u$ pour tenir compte de la taille de $U$
et la variable $v$ pour tenir compte de la taille de $V$.

\medskip

On va maintenant se servir de la même description récursive des intervalles de
$\Tam_n$ que dans \cite{chapoton_SLC}, auquel le lecteur pourra se reporter pour
davantage de détails.

On dit qu'un intervalle $(S,T)$ est \textit{indécomposable} lorsque
$S$ est indécomposable. Tout intervalle admet une unique écriture sous
la forme 
\begin{equation*}
(S_0/S_1/\dots/S_{\ell}, T_0/T_1/\dots/T_{\ell})
\end{equation*}
où $(S_i,T_i)$ est un intervalle indécomposable pour tout $i$. Ceci
résulte du fait que la composition $\comp(S)$ est plus grossière que
$\comp(T)$. Dans cette situation, l'intervalle $(S,T)$ est isomorphe au produit cartésien des intervalles $(S_i,T_i)$.

Tout intervalle indécomposable $(S,T)$ dans $\Tam_n$ pour $n \geq 2$
s'obtient de manière unique à partir d'un intervalle $(S',T')$ dans
$\Tam_{n-1}$ par la procédure suivante. L'arbre $S$ s'obtient en
ajoutant à $S'$ une feuille qui vient s'attacher à gauche sur l'arête
située sous la racine de $S'$. On note le résultat $Y \backslash
S'$. L'arbre $T$ s'obtient en ajoutant à $T'$ une feuille qui vient
s'attacher à gauche sur une des arêtes du bord gauche de $T'$. On note
$Y * T'$ l'ensemble des arbres obtenus ainsi. Si $a$ est une arête du
bord de $T'$, one note $Y *_a T'$ l'arbre obtenu en ajoutant à gauche
une feuille sur l'arête $a$. Le nombre d'intervalles $(S,T)$ ainsi
construits à partir d'un intervalle $(S',T')$ fixé est donc égal au
nombre d'arêtes du bord gauche de $T'$.

Ces deux décompositions donnent une description récursive complète des
intervalles de Tamari. Il reste à comprendre le comportement des
paramètres qui nous intéressent dans ces décompositions.

\medskip

On utilise deux séries génératrices :
\begin{itemize}
\item la série $\Theta$ pour les intervalles indécomposables,

\item la série $\Phi$ pour tous les intervalles.
\end{itemize}

On va utiliser la description alternative suivante des statistiques
contrôlées par les variables $x,y,\yb$, qui résulte directement de la
description des couvertures dans $\Tam_n$ par la rotation des arbres
binaires plans. Pour un intervalle $(S,T)$, la puissance de de $y$ est
le nombre d'arêtes internes de $T$ de direction $/$ et celle de $\yb$
est le nombre d'arêtes internes de $S$ de direction $\backslash$. La
puissance de $x$ est plus subtile, bornée par le nombre d'arêtes internes
de $S$ de direction $/$, mais sous la condition que la rotation
de $S$ associée à cette arête donne un arbre inférieur à $T$.

\begin{proposition}
On a les équations fonctionnelles :
\begin{equation}
  \label{eq_Phi}
    \Phi(u,v) = \Theta(u,v) + \yb \Phi(v,v) \Theta(u,v) / v,
\end{equation}
et
\begin{multline}
  \label{eq_Theta}
  \Theta(u,v) = t v \Big{(}u +
                  y u \frac{\Phi(u,1)-\Phi(1,1)}{u - 1} + \\
                  x y u \frac{\Phi(u,u)-\Phi(u,1)}{u - 1} +
                  (x - x y) \Phi(u,u) \Big{)}.
\end{multline}
\end{proposition}

\begin{proof}
  La première équation \eqref{eq_Phi} s'obtient par découpage d'un
  intervalle $(S,T)$ le long du bord gauche. Un
  intervalle est soit indécomposable, soit de la forme
  $(S_0/S_1,T_0/T_1)$ avec $(S_0,T_0)$ un intervalle et $(S_1,T_1)$ un
  intervalle indécomposable.

  Lorsque l'intervalle $(S,T)$ est indécomposable, on a le premier
  terme du membre de droite de \eqref{eq_Phi}.

  Sinon, on ajoute une arête interne d'orientation $\backslash$ le
  long des bords gauches de $S$ et de $T$. On obtient donc un facteur
  $\yb$ supplémentaire, et il n'y a pas de facteur $x$ ou $y$ à
  ajouter. En ce qui concerne le découpage du bord gauche de $T$ en
  deux parties $U$ et $V$, il doit nécessairement se produire dans le
  facteur $T_1$. On doit aussi diviser par $v$ pour tenir compte de la
  greffe, qui supprime une arête du bord gauche. Ceci donne le second terme du membre de droite de
  \eqref{eq_Phi}.

  La preuve de la seconde équation \eqref{eq_Theta} est plus
  compliquée. Le premier terme $t u v$ correspond à l'unique intervalle dans
  $\Tam_1$, qui est bien indécomposable. Les trois autres termes
  correspondent à la distinction de plusieurs cas pour les intervalles indécomposables
  $(S,T)$ avec $S=Y \backslash S'$ et $T$ dans $Y * T'$. Il faut
  étudier les arêtes internes nouvelles dans $Y\backslash S'$ et $Y*T'$
  avec soin, en fonction de l'endroit où se greffe la nouvelle arête
  sur le bord gauche de $T'$.

  On peut d'abord comprendre le rôle de la variable $v$ en facteur. Comme
  $S$ est indécomposable, la composition $\comp_S$ a une seule
  part. Seule l'arête supérieure du bord gauche de $T$ est donc dans
  $V$, ce qui correspond au facteur $v$ dans le membre de droite de \eqref{eq_Theta}.

  La construction de $S$ et $T$ à partir de $S'$ et $T'$ ajoute des
  arêtes internes. Dans $S$, on obtient toujours une arête interne
  d'orientation $/$ en plus sous la racine de $S'$. Elle peut contribuer éventuellement à un
  facteur $x$ (voir ci-après), mais aucun facteur $\yb$
  n'intervient. Dans $T$, on ajoute en général une arête interne
  d'orientation $/$, sauf lorsque la nouvelle feuille vient se greffer
  tout en haut du bord gauche de $T'$ qui donne une arête interne
  d'orientation $\backslash$ en plus. On aura donc un facteur $y$ pour
  tous les arbres $T$ dans $Y * T'$ sauf pour l'arbre $Y / T'$. Cette
  distinction donne lieu à un terme correctif, le dernier terme dans le
  membre de droite de \eqref{eq_Theta}.

  Il reste à comprendre le comportement de la variable $x$. C'est le
  point le plus subtil. On se donne donc un intervalle indécomposable
  $(S,T)$ avec $S=Y \backslash S'$ et $T$ obtenu par ajout d'une
  feuille à $T'$, qui se greffe sur l'arête $a$ du bord gauche de
  $T'$. 

  Si $T$ est un arbre et $a$ une arête du bord gauche de $T$, on note
  $k_a(T)$ le nombre de sommets dans le sous-arbre de $T$ (contenant la racine de $T$) obtenu en
  coupant l'arête $a$.

  On a besoin de la description précise suivante de l'intervalle
  $(S,T)$ (implicite dans \cite[\S 4]{chapoton_SLC}). Les éléments de
  cet intervalle sont exactement les arbres $R = Y *_b R'$ obtenus à partir d'un
  élément $R'$ dans l'intervalle $(S',T')$ en ajoutant une feuille qui
  se greffe sur l'arête $b$ du bord gauche de $R'$, sous la
  condition que $k_b(R') \leq k_a(T')$. Les couvertures dans
  l'intervalle $(S,T)$ sont de deux types. Le premier type de
  couverture (type \rm{I}) est un changement de l'arête $b$ vers
  l'arête $b'$ située juste en dessus, si $k_{b'}(R') \leq k_a(T')$. Le
  second type de couvertures (type \rm{II}) provient d'une couverture
  $R' \couvert R''$ dans l'intervalle $(S',T')$. Une telle couverture
  peut modifier le coté gauche de $R'$, mais il existe toujours une unique
  arête $c$ du bord gauche de $R''$ telle que $k_{c}(R'') =
  k_b(R')$. Cette arête $c$ donne la couverture
  $(Y *_b R') \couvert (Y *_c R'')$ dans $(S,T)$.

  En utilisant cette description, on peut comprendre quelles
  couvertures de $S$ restent inférieures à $T$, parmi toutes les
  couvertures associées aux arêtes $\backslash$ de $S$. Ce sont d'une
  part des couvertures provenant directement des couvertures de $S'$
  qui restent inférieures à $T'$ (type \rm{II}) et d'autre part les
  couvertures de $S$ où la rotation se produit en la racine de $S$
  (type $I$). Si on effectue une telle rotation vers un arbre dans
  $(S,T)$, on voit que la valeur correspondante de $k_b(R')$ est égale
  au dernier terme $k(S')$ de la composition $\comp(S')$, ce qui est
  possible si est seulement si $k(S') \leq k_a(T')$.

  On voir donc qu'une nouvelle couverture de type $x$ apparaît (en
  plus des couvertures de ce type héritées de $(S',T')$) si et
  seulement si la construction de $T$ se fait en attachant une arête
  $a$ dans la partie $V$ du bord de $T'$. Cette distinction donne
  lieu aux deux termes avec division par $u-1$ dans le membre de
  droite de \eqref{eq_Theta}, qui correspondent aux cas $a \in U$ et
  $a \in V$.
\end{proof}

\begin{remark}
Si on spécialise le système \eqref{eq_Phi}, \eqref{eq_Theta} en $x=1$, on peut passer à un seul paramètre
catalytique (en identifiant $u$ et $v$, on retombe sur le paramètre
catalytique usuel pour les intervalles de Tamari). Par contre, si on
spécialise en $y=1$, on a quand même besoin de $u$ et $v$, mais c'est une
équation du même type un peu plus simple. On peut aussi spécialiser en $\yb=1$
sans aucune difficulté.
\end{remark}

\begin{remark}
À la place de \eqref{eq_Phi}, on peut utiliser
\begin{equation}
  \label{eq_Phi_alternative}
    \Phi(u,v) = \Theta(u,v) + \yb \Theta(v,v) \Phi(u,v) / v,
\end{equation}
qui se démontre facilement de manière combinatoire.
\end{remark}

\begin{remark}
En fait, on peut se débarrasser de la variable $v$ en faisant d'une part $v=1$ et d'autre part $v=u$ dans les équation \eqref{eq_Phi} et \eqref{eq_Theta}. Ça permet de se ramener à un système de deux équations en $\Phi(u,u)$, $\Phi(u,1)$ et $\Phi(1,1)$ dont
\begin{equation}
  (u + \yb \Phi(u,u)) \Phi(u,1) = \Phi(u,u) (1 + \yb \Phi(1,1)).
\end{equation}
Ceci sert de point de départ pour montrer l'algébricité, voir la section
\ref{algeb}.
\end{remark}

% Les premiers termes de $\Phi$ et $\Theta$ sont (avec $\yb=1$):
% \begin{align}
% \Phi &= u vt + \left(u^{2} v x + u v^{2} + u v y\right)t^{2} \\
%        &+ \left(u^{3} v x^{2} + u^{3} v x + u^{2} v^{2} x + u v^{3} x + 3 u^{2} v x y + u v^{3} + 2 u v^{2} y + u v x y + u v y^{2} + u v y\right)t^{3} + O(t^{4})\\
% \Theta &= u vt + \left(u^{2} v x + u v y\right)t^{2} \\&+ \left(u^{3} v x^{2} + u^{3} v x + 3 u^{2} v x y + u v x y + u v y^{2} + u v y\right)t^{3} + O(t^{4})
% \end{align}
% et en $(u,v)=(1,1)$ et $\yb=1$:
% \begin{align}
% \Phi(1,1) &= t + \left(x + y + 1\right)t^{2} + \left(x^{2} + 4 x y + y^{2} + 3 x + 3 y + 1\right)t^{3} + O(t^{4})\\
% \Theta(1,1) &= t + \left(x + y\right)t^{2} + \left(x^{2} + 4 x y + y^{2} + x + y\right)t^{3} + O(t^{4})
% \end{align}

Les premiers termes de $\Phi(u,v)$ et $\Theta(u,v)$ sont :
\begin{align*}
\Phi(u,v) &= u vt + \left(u^{2} v x + u v^{2} \yb + u v y\right)t^{2} \\&+ \big{(}u^{3} v x^{2} + u^{3} v x \yb + u^{2} v^{2} x \yb + u v^{3} x \yb + u^{2} v x y \yb + u v^{3} \yb^{2} + 2 u^{2} v x y \\ & + 2 u v^{2} y \yb + u v x y + u v y^{2} + u v y \yb\big{)}t^{3} + O(t^{4})\\
\Theta(u,v) &= u vt + \left(u^{2} v x + u v y\right)t^{2} \\&+ \left(u^{3} v x^{2} + u^{3} v x \yb + u^{2} v x y \yb + 2 u^{2} v x y + u v x y + u v y^{2} + u v y \yb\right)t^{3} + O(t^{4})
\end{align*}
et en $(u,v)=(1,1)$ :
\begin{align*}
\Phi(1,1) &= t + \left(x + y + \yb\right)t^{2} + \left(x y \yb + x^{2} + 3 x y + y^{2} + 3 x \yb + 3 y \yb + \yb^{2}\right)t^{3} + O(t^{4})\\
\Theta(1,1) &= t + \left(x + y\right)t^{2} + \left(x y \yb + x^{2} + 3 x y + y^{2} + x \yb + y \yb\right)t^{3} + O(t^{4})
\end{align*}

On va montrer dans la section \ref{algeb} qu'une fois oubliés les
paramètres catalytiques $u$ et $v$, la série $\Phi(1,1)$ est totalement
symétrique en $x$, $y$ et $\yb$.

\subsection{$q$-analogue}

Un autre paramètre important sur les intervalles de Tamari est la
longueur de la plus grande chaîne entre le minimum et le maximum d'un
intervalle, c'est-à-dire le nombre maximal de rotations nécessaires
pour passer du minimum au maximum. Ce paramètre est aussi appelé le
nombre d'inversions de cet intervalle.

On peut facilement ajouter le paramètre $q$ pour la longueur de la
plus grande chaîne dans les équations \eqref{eq_Phi} et
\eqref{eq_Theta} avec $x,y,\yb,u,v$.

L'équation \eqref{eq_Phi} pour $\Phi$ reste inchangée, car la
puissance de $q$ est additive pour ce type de décomposition. Pour
l'équation \eqref{eq_Theta}, on trouve
\begin{multline}
  \label{eq_Theta_q}
  \Theta(u,v) = t v \Big{(}u +
                  y u \frac{\Phi(qu,1)-\Phi(1,1)}{qu - 1} + \\
                  x y u \frac{\Phi(qu,qu)-\Phi(qu,1)}{qu - 1} +
                  \frac{x - x y}{ q} \Phi(qu,qu) \Big{)}.
\end{multline}

Par auto-dualité du poset de Tamari, $q$ a la même double-distribution avec $\yb$ et $y$. La double-distribution avec $x$ ou $\xb$ est différente.

\subsection{Algébricité et symétrie}

\label{algeb}

On part du système suivant, forme équivalente au système de $4$ équations obtenues en faisant $v=1$ et $v=u$ dans \eqref{eq_Phi}, \eqref{eq_Theta}, 
suivi d'une élimination des $\Theta$ :

\begin{equation}
  \Phi_{uu} (1 + \yb \Phi_{11}) = \Phi_{u1} (u + \yb \Phi_{uu})
\end{equation}
\begin{equation}
  \Phi_{u1} = t (1 + \yb \Phi_{11}) \left(u+yu\frac{\Phi_{u1}-\Phi_{11}}{u-1}+xyu\frac{\Phi_{uu}-\Phi_{u1}}{u-1}+(x-xy)\Phi_{uu}\right)
\end{equation}
avec des notations abrégées pour les diverses spécialisations des
paramètres $u$ et $v$ dans la série $\Phi$.

On élimine alors $\Phi_{u1}$ puis on factorise, pour obtenir une
équation de type catalytique standard (avec une seule variable
catalytique) pour $\Phi_{uu}$  de la forme
$P(\Phi_{uu},\Phi_{11},t,u) = 0$ pour un certain polynôme $P$.

On écrit alors, en suivant \cite{BM_Jehanne}, le système de trois
équations algébriques
\begin{align*}
  P &= 0,\\
  \partial_{\Phi_{uu}}P&=0,\\
  \partial_u P&=0.
\end{align*}

Par élimination de $\Phi_{uu}$ et $u$, ce système nous donne une
équation algébrique de degré $7$ en $\Phi_{11}$. L'élimination directe
étant un peu difficile, on commence par éliminer $\Phi_{uu}$ pour
obtenir deux polynômes, puis on utilise leur résultant en $u$, qu'on
doit factoriser pour obtenir la bonne équation.

On obtient ainsi une (grosse) équation polynomiale, de degré $7$ en
$\Phi_{11}$ et de degré $6$ en $t$, comportant $1478$ monômes et des
coefficients compris entre $-628$ et $628$. Son polytope de Newton a
$109$ sommets. Le terme dominant en $\Phi_{11}$ a pour coefficient
\begin{equation}
  (x - 1) (y - 1) (\yb - 1) x^{5} y^{5} \yb^{5} t^{6}.
\end{equation}

On constate que cette équation est totalement symétrique en $x,y,\yb$
et qu'elle s'écrit sous la forme
\begin{equation}
  \Phi_{11} = t + \dots
\end{equation}
où les termes omis sont divisibles par $t$ et ont degré total au moins $2$ par rapport aux
variables $(t,\Phi_{11})$. Il s'ensuit que cette équation a une unique
solution qui est une série formelle en $t$, donc cette solution hérite
de la symétrie ternaire. Ceci démontre le théorème \ref{principal}. Il
reste la question de trouver des involutions qui démontrent la
symétrie de manière purement combinatoire.

%(cf. involution de Viviane ?)
%A regarder sur de petits exemples ?

\subsection{Intervalles synchrones}

\label{preuve_synchrone}

On donne ici la preuve du théorème \ref{synchrone}.

\begin{proof}
  D'abord, on a bien une inclusion des intervalles de degré
  $n-1$ en $y,\yb$ dans les intervalles synchrones. En effet, imposer
  la restriction sur le degré en $y,\yb$ revient, au vu des équations
  fonctionnelles \eqref{eq_Phi} et \eqref{eq_Theta}, à exclure le
  terme $x \Phi(u,u)$ dans \eqref{eq_Theta}. Au niveau combinatoire,
  ceci correspond à exclure exactement le seul cas, dans la
  construction inductive des intervalles, où on peut sortir des
  intervalles synchrones : lorsqu'on rajoute sur le bord gauche du
  maximum une feuille tout au sommet.

  Pour montrer une égalité des cardinaux, il suffit de résoudre
  les équations fonctionnelles restreintes. Après simplification en
  posant $x=1$, $y=1$, $\yb=1$ et $v=u$, celles-ci sont
  \begin{equation*}
    \Phi'(u,u) = \Theta'(u,u) + \Phi'(u,u) \Theta'(u,u) / u,
  \end{equation*}
  et
  \begin{equation*}
    \Theta'(u,u) = t u \Big{(}u +
    u \frac{\Phi'(u,u)-\Phi'(1,1)}{u - 1}
    - \Phi'(u,u) \Big{)}.
  \end{equation*}
  On en déduit par les méthodes de \cite{BM_Jehanne} l'équation algébrique
  \begin{equation*}
    \Phi_{11}^{'3} t^{2} + 6 \Phi_{11}^{'2} t^{2} + 2 \Phi_{11}^{'2} t + 12 \Phi'_{11} t^{2} - 10 \Phi'_{11} t + 8 t^{2} + \Phi'_{11} -  t,
  \end{equation*}
  qu'il est facile de comparer avec l'équation connue pour la
  série génératrice commune des permutations triables-par-deux-piles et des intervalles
  synchrones (voir \href{http://oeis.org/A000139}{A139}).
\end{proof}

\section{Diverses propriétés}

On regroupe ici diverses propriétés des intervalles en général et des
intervalles dans $\Tam_n$ en particulier, ainsi que quelques questions ouvertes.

\begin{lemma}
  Soit $P$ un poset fini. Un intervalle $(u,v)$ est de degré $0$ en
  $x$ ou en $\xb$ si et seulement si cet intervalle est réduit à un
  point, \textit{i.e.} $u=v$.
\end{lemma}
\begin{proof}
  Si l'intervalle réduit à un élément, il n'existe clairement aucune
  arête de type $x$ ou $\xb$ dans le diagramme de
  Hasse. Réciproquement, si l'intervalle n'est pas réduit à un point, on
  peut trouver un chemin de $u$ à $v$, donc des couvertures de type
  $x$ et $\xb$.
\end{proof}

\begin{lemma}
  Soit $P$ un poset fini. Un intervalle $(u,v)$ est de degré $0$ en
  $y$ (resp. en $\yb$) si et seulement si $v$ est un élément maximal du poset $P$ (resp. $u$ un élément minimal)
\end{lemma}
\begin{proof}
  C'est immédiat, le degré en $y$ étant simplement le nombre d'arêtes
  sortantes de $v$ dans le diagramme de Hasse de $P$. Idem pour le
  degré en $\yb$ et les arêtes entrantes en $u$.
\end{proof}

\begin{lemma}
  Pour le treillis de Tamari $\Tam_n$, le degré de $\quat_n$ par
  rapport à $(x,\yb)$ ou $(y,\xb)$ est au plus $n-1$. C'est aussi vrai pour le degré par rapport à $(x,y)$, $(\xb,\yb)$ et $(y,\yb)$. Le degré de
  $\quat_n$ par rapport à chacune des variables $x,y,\yb,\xb$ est au
  plus $n-1$. 
\end{lemma}
\begin{proof}
  Le premier point résulte du fait bien connu que le diagramme de
  Hasse de
  $\Tam_n$ est le graphe des sommets et des arêtes de l'associaèdre de
  dimension
  $n-1$, qui est un polytope régulier.  Le second point résulte du
  premier et du théorème \ref{principal}.  Le dernier point est une
  conséquence du premier.
\end{proof}

\begin{proposition}
  \label{bicubique}
  Pour le treillis de Tamari $\Tam_n$, le degré en $(x,y,\yb)$ d'un
  monôme de $\quat_n$ est au moins $n-1$. Le cardinal de l'ensemble
  des intervalles de degré $n-1$ en $(x,y,\yb)$ est le nombre de
  cartes enracinés bicubiques, donné par la suite
  \href{http://oeis.org/A000257}{A257}.
\end{proposition}
\begin{proof}
  La borne inférieure s'obtient par récurrence en utilisant la forme
  des équations fonctionnelles \eqref{eq_Phi} et \eqref{eq_Theta} pour
  montrer simultanément le même énoncé pour $\Phi$ et $\Theta$.

  Pour obtenir le cardinal, on réduit les équations fonctionnelles en
  imposant la contrainte voulue sur le degré et en posant $x=1$, $y=1$ et
  $\yb=1$ :
  \begin{equation}
    \Phi''(u,v) = \Theta''(u,v) + \Phi''(v,v) \Theta''(u,v) / v,
  \end{equation}
  et
  \begin{equation}
    \Theta''(u,v) = t v \Big{(}u +
    u \frac{\Phi''(u,1)-\Phi''(1,1)}{u - 1} + \Phi''(u,u) \Big{)}.
  \end{equation}
  Pour en déduire une équation algébrique pour $\Phi''(1,1)$, on procède
  comme dans la section \ref{algeb}. On spécialise ces deux équations
  en $v=1$ et $v=u$, obtenant ainsi $4$ équations. On élimine
  $\Theta''(u,1)$, $\Theta''(u,u)$ et $\Phi''(u,1)$ pour obtenir une
  équation catalytique usuelle sous la forme d'un polynôme en
  $\Phi''(u,u),\Phi''(1,1),t,u$. On en déduit par les méthodes de \cite{BM_Jehanne}, l'équation algébrique
  \begin{equation}
    16 \Phi_{11}^{''2} t^{2} + 24 \Phi''_{11} t^{2} - 12 \Phi''_{11} t + 9 t^{2} + \Phi''_{11} -  t,
  \end{equation}
  qu'on peut aisément comparer à celle connue pour la série
  génératrice de la suite \href{http://oeis.org/A000257}{A257}.
\end{proof}
%Meme nombre que les intervalles modernes (A257), mais pas meme
%  distribution de la stat $(x,y,\yb,\xb)$, donc pas pareil.

\begin{remark}
  On peut aussi obtenir par une restriction similaire des équations
  fonctionnelles pour l'ensemble des intervalles qui sont à la fois de
  degré $n-1$ en $(y,\yb)$ et de degré $n-1$ en $(x,\yb)$ .
\end{remark}

\subsection{Questions ouvertes sur la distribution en $\{x,y,\yb,\xb\}$}

Question : peut-on écrire une équation catalytique pour la série complète en $x,y,\yb,\xb$ ?

\begin{conjecture}
Les seuls intervalles de degré  $n-1$ en $(x,y,\yb,\xb)$ sont les intervalles simples $(S,S)$. 
\end{conjecture}

Les intervalles simples ont clairement ce degré, parce que le diagramme de Hasse est un graphe régulier. 

\medskip

\begin{conjecture}
  Le nombre d'intervalles de degré $n-1$ en $(x,y)$ et de degré $n-1$
  en $(\xb,\yb)$ est un nombre de Motzkin (\href{http://oeis.org/A001006}{A1006}).
\end{conjecture}
L'ensemble d'intervalles correspondant contient clairement
l'intervalle formé par $\Tam_n$ tout entier. Il semble que le poset
induit par $\Int(\Tam_n)$ sur cet ensemble d'intervalles soit une
antichaine.

\begin{proposition}
   Le nombre d'intervalles de degré $n-1$ en $(x,\yb)$ et de degré  $n-1$ en $(\xb,y)$ est un nombre de Motzkin.
\end{proposition}
Cet énoncé est équivalent au précédent modulo la conjecture \ref{conjecture_x_xbar}.

% Tout ça donne beaucoup de triangles inconnus dans OEIS.

% Petit aparté sur le coté singularité:

% Si on regarde la singularité associée au polynôme $D_n(a, \ab)$
% il semble que le nombre de Milnor total soit $2 n (n + 1) + 1$
% avec contribution de la singularité en $(a,\ab)=(0,0)$ pour $n^2$.

\subsection{Racines réelles négatives}

Parmi les manières de spécialiser en $1$ certaines des variables dans
$\quat_n$ pour obtenir un polynôme en une variable, il semble s'en
trouver trois types (modulo les symétries de $\quat_n$) qui donnent
des polynômes ayant seulement des racines réelles négatives : ce sont
$\quat_n(z,1,1,1)$, $\quat_n(z,z,1,1)$, $\quat_n(z,z,z,1)$.

On peut aussi extraire d'autres polynômes en une variable par
restriction à des sous-ensembles. On obtient notamment ainsi les
polynômes de Narayana \href{http://oeis.org/A001263}{A1263}, dont on
sait que les racines sont réelles négatives (voir \cite[\S 5]{branden}
pour des références sur ce résultat).

On peut aussi obtenir un polynôme qui compte les permutations
triables-par-deux-piles selon un paramètre, voir \href{http://oeis.org//A082680}{A82680}. Le fait que ces polynômes ont seulement des zéros réels négatifs a été démontré par Brändén \cite[Th. 5.1]{branden}.

\section{Relation avec les statistiques de canopée}

\label{canopee}

Une autre classe de statistiques naturelles sur les intervalles de
Tamari est fournie par la description de leurs canopées.

À un intervalle de Tamari, on peut associer un mot en trois ``double-lettres'' $\LL$, $\LR$ et $\RR$. Dans ce mot, la lettre en position
$i$ décrit les orientations des feuilles en position $i$ dans le
minimum et le maximum de l'intervalle. Par exemple, l'intervalle
\begin{center}
  \label{arbre}
  \includegraphics[height=2cm]{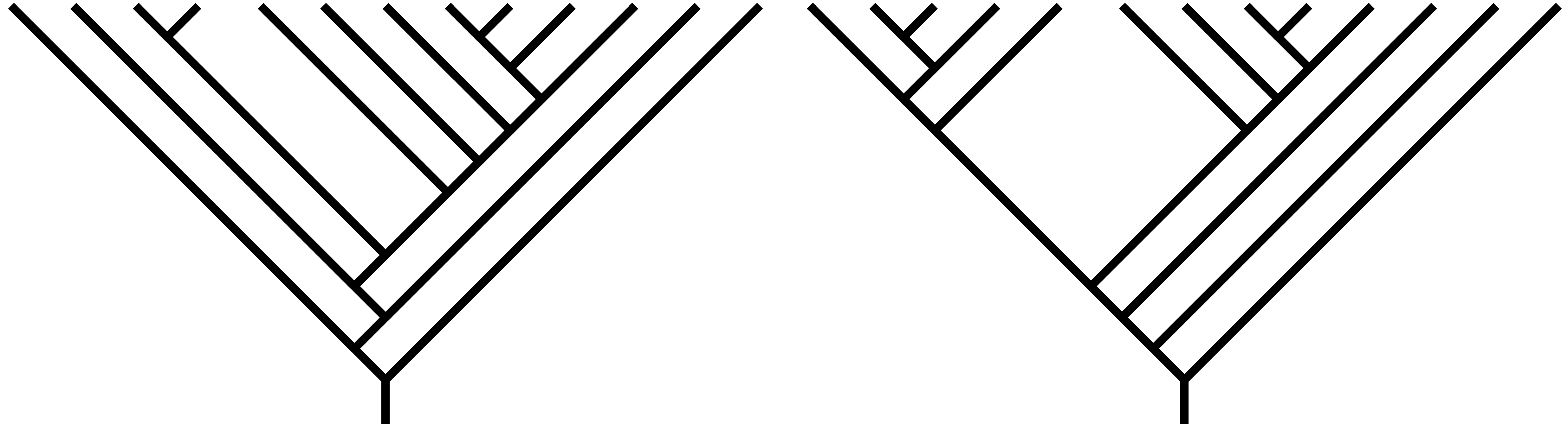}
\end{center}
a pour mot associé
$(\LL,\LL,\LR,\RR,\LR,\LL,\LL,\LL,\RR,\RR,\RR,\RR,\RR)$. On peut
montrer (en considérant l'action de la rotation sur les
canopées) que seules les trois combinaisons $(L,L),(L,R)$ et $(R,R)$
sont possibles, d'où l'emploi de ces trois double-lettres.

Par retournement des arbres binaires plans et des intervalles, on a
une symétrie évidente d'ordre $2$ qui échange $\LL$ avec $\RR$. La
symétrie d'ordre $3$ en $x,y,\yb$ est liée à cette symétrie d'ordre
$2$ par la relation suivante.

\begin{proposition}
  La double-distribution selon les variables $y$ et $\yb$ coïncide
  avec la double-distribution selon les variables $\LL$ et $\RR$.
\end{proposition}
\begin{proof}
Pour les trois paramètres $\LL$, $\LR$ et $\RR$ décrivant les
canopées, on démontre comme dans la section \ref{series_gens}, en utilisant la même décomposition introduite dans \cite{chapoton_SLC}, les équations
\begin{equation}
  \Phi(u) = \Theta(u) + \Phi(u) \Theta(u) / (u \LL),
\end{equation}
et
\begin{equation}
  \Theta(u) = t u \left(\LL\RR u +
    u \LL \frac{\Phi(u)- \Phi(1)}{u - 1} +
    (\LR-\LL)\Phi(u) \right).
\end{equation}
Comme les coefficients de ces séries sont homogènes, on peut sans perte d'information remplacer $\LR$ par $1$. On peut aussi diviser tous les termes par $\LL\,\RR$. Les équations obtenues sont
\begin{equation}
  \label{eq_Phi_canopee}
  \Phi(u) = \Theta(u) + \RR \Phi(u) \Theta(u) / u,
\end{equation}
et
\begin{equation}
  \label{eq_Theta_canopee}
  \Theta(u) = t u \left(u +
    u \LL \frac{\Phi(u)- \Phi(1)}{u - 1} +
     (1-\LL)\Phi(u) \right).
\end{equation}
On voit sans difficulté que ces équations sont identiques à la
spécialisation de \eqref{eq_Phi} et \eqref{eq_Theta} en $x=1$ et $v=u$, en
identifiant les variables $\LL=y$ et $\RR = \yb$.

On a donc obtenu la relation voulue entre la double-statistique
$(y,\yb)$ et la double statistique $(\LL,\RR)$.
\end{proof}

Voici un tableau de cette double-distribution pour $n=1,2,3,4,5$ :
\begin{equation*}
\left(\begin{array}{r}
1
\end{array}\right)\left(\begin{array}{rr}
1 & 0 \\
1 & 1
\end{array}\right)\left(\begin{array}{rrr}
1 & 0 & 0 \\
3 & 4 & 0 \\
1 & 3 & 1
\end{array}\right)\left(\begin{array}{rrrr}
1 & 0 & 0 & 0 \\
6 & 10 & 0 & 0 \\
6 & 21 & 10 & 0 \\
1 & 6 & 6 & 1
\end{array}\right)\left(\begin{array}{rrrrr}
1 & 0 & 0 & 0 & 0 \\
10 & 20 & 0 & 0 & 0 \\
20 & 81 & 49 & 0 & 0 \\
10 & 65 & 81 & 20 & 0 \\
1 & 10 & 20 & 10 & 1
\end{array}\right)
\end{equation*}

\begin{remark}
  On ne peut pas espérer une expression simple sous forme de produit
  de coefficients binomiaux pour les coefficients de cette double-distribution, car certains d'entre eux comportent de grands facteurs
  premiers (un coefficient dans $[t^9]\Phi$ vaut $84089$). Il en est
  de même pour les coefficients de la triple-distribution en $x,y,\yb$ avec
  par exemple un coefficient $18691$ dans $[t^9]\Phi$.
\end{remark}

\bibliographystyle{alpha}
\bibliography{bilan_double_catalyse}

\end{document}